\documentclass[letterpaper, 10 pt, conference]{ieeeconf}  

\usepackage{Preamble}

\title{\LARGE \bf
Radially Symmetric Mean-Field Games with Congestion
}

\author{David Evangelista, Diogo A. Gomes and Levon Nurbekyan
}
\begin{document}

\maketitle
\thispagestyle{empty}
\pagestyle{empty}

\begin{abstract}
	Here, we study radial solutions for first- and second-order stationary Mean-Field Games (MFG) with congestion on $\Rr^d$.
	MFGs with congestion model problems where the agents' motion is hampered in high-density regions. The radial case, which is one of the simplest non one-dimensional MFG, is relatively tractable. As we observe in this paper, the Fokker-Planck equation is integrable with respect to one of the unknowns. Consequently, we obtain a single  equation substituting this solution into the Hamilton-Jacobi equation.   
	For the first-order case, we derive explicit formulas; for the elliptic case, we study a variational formulation of the resulting equation.
	In both cases, we use our approach to compute numerical approximations to the solutions of the corresponding MFG systems.
%
%
%
%
%
%
%
%
%
%
\end{abstract}

\section{Introduction}

Mean-field games are models for large populations of competing rational agents
that were introduced in \cite{Caines1} and \cite{Caines2} and, independently, around the same time in \cite{ll1}, \cite{ll2}, and \cite{ll3}. 
 Usually, these games are determined by a system of a Hamilton-Jacobi equation coupled with a Fokker-Planck equation. Partial differential equations (PDE), and in particular systems, seldom admit explicit solutions. Thus, PDE theory focuses on matters such as the uniqueness, 
 existence, and regularity of solutions, see, for example, \cite{GPV}. 
 For MFGs, few explicit solutions are known in spite of their interest 
 in understanding the qualitative properties of the models and for the validation of numerical methods. A discussion on explicit solutions (as well as appropriate references) can be found in  \cite{GPV}. Some recent 
 results for one-dimensional problems are presented in \cite{GNPr'16}, \cite{GNPr2'16} and \cite{nurbekyan'17}. 

Here, we consider MFG models for which the Fokker-Planck equation is integrable with respect to one of the unknowns. This integrability reduces the system to a single equation that often can be solved explicitly.
More precisely, we study the following models.


\begin{problem}[First order with congestion]\label{P1}
	Assume that $0\leq \alpha<2$. Moreover, let $V\in C^{\infty}(\Omega)$ and $g \in C^{\infty}(\Rr^+)$ be given potential and coupling, respectively. Then, find $(u,m,\Hh)\in C^{\infty}(\Omega)\times C^{\infty}(\Omega) \times \Rr$ such that
	\begin{equation}\label{mainfo}
	\begin{cases}
	\frac{|Du|^2}{2 m^\alpha}+ V(x)=  g(m)+\Hh & \mbox{ in } \Omega \\
	-\div(m^{1-\alpha}Du)=0 \quad &\mbox{ in } \Omega\\
	m>0,\ \int\limits_{\Omega} m(x)dx=1.
	\end{cases}
	\end{equation}
\end{problem}
\begin{problem}[Second order with congestion]\label{P2}
	Assume that $0\leq \alpha<2$. Moreover, let $V\in C^{\infty}(\Omega)$ and $g \in C^{\infty}(\Rr^+)$ be given potential and coupling, respectively. Then, find $(u,m,\Hh)\in C^{\infty}(\Omega)\times C^{\infty}(\Omega) \times \Rr$ such that
	\begin{equation}\label{mainso}
	\begin{cases}
	-\Delta u+\frac{|Du|^2}{2 m^\alpha}+ V(x)=  g(m)+\Hh & \mbox{ in } \Omega \\
	-\Delta m-\div(m^{1-\alpha}Du)=0 \quad &\mbox{ in } \Omega\\
	m>0,\ \int\limits_{\Omega} m(x)dx=1.
	\end{cases}
	\end{equation}
\end{problem}
In Problems \eqref{P1} and \eqref{P2}, we consider $\Omega=\Rr^d$ or $\Omega=\Rr^d \setminus\{0\}$ for $d\geq 2$. The properties of solutions of previous problems strongly depend on the choice of $\Omega$.

MFGs with congestion arise when agents face increasing costs in moving with an increasing density. For example, the Hamilton-Jacobi equation in \eqref{mainfo}
corresponds to the following control problem:
\begin{align*}
u(x)
&=
\inf\Big[
\int_0^T 
\Big(
\frac{m^\alpha(\bx(t))|\bdx(t)|^2}{2}\\&\qquad\quad
-V(\bx(t)) +g(m(\bx(t)))+\Hh\Big)dt\\&\quad+u(\bx(T))\Big], 
\end{align*}
where the infimum is taken over all Lipschitz trajectories, $\bx$ with 
$\bx(0)=x$. The preceding variational principle is a fixed-point
problem for $u$ because we are looking for a stationary solution rather than a 
time-dependent.  

The congestion problem was introduced in \cite{LCDF}.
The existence for stationary MFGs  with congestion, positive viscosity, and a quadratic Hamiltonian was solved in \cite{GMit} and, subsequently, this problem was examined in more generality in \cite{GE16}. The time-dependent case was considered in \cite{GVrt2} (classical solutions) 
and \cite{Graber2} (weak solutions). Later, \cite{AP} examined weak solutions for time-dependent problems. 
Apart from the results in \cite{FG2} using monotonicity techniques and the one-dimensional examples in \cite{GNPr2'16} and \cite{nurbekyan'17}, little is known for first-order MFGs with congestion. 

Radial MFGs arise when there is rotational symmetry with respect to a 
central point, here the origin. For monotone MFGs, where there is
uniqueness of solution, the unique stationary solution is radial. 
Moreover, we expect stationary solutions to encode the long-time behavior of MFGs. More precisely, the time-dependent problem associated with Problem \ref{P1} is 
\begin{equation}\label{mainfotd}
\begin{cases}
-u_t+\frac{|Du|^2}{2 m^\alpha}+ V(x)=  g(m)+\Hh & \mbox{ in } \Omega \\
m_t-\div(m^{1-\alpha}Du)=0 \quad &\mbox{ in } \Omega\\
m>0,\ \int\limits_{\Omega} m(x,t)dx=1.
\end{cases}
\end{equation}
The previous system is supplemented with initial-terminal conditions
\begin{equation}
\begin{cases}
u(x,T)=\mu(x)\\
m(x,-T)=\eta(x).
\end{cases}
\end{equation}
Then, we expect that the effect
upon the solution at time zero, $(u(x,0), m(x,0))$,
from the initial distribution, $\eta$, and from the terminal cost, $\mu$,
fades as $T\to \infty$. Thus, we expect $(u(x,0), m(x,0))$ to converge to a stationary solution. Some results in this direction can be found in 
the finite-state case in \cite{GMS2} and in the continuous case in \cite{CLLP}
and \cite{cllp13}. For the convergence to hold, the MFG must satisfy 
monotonicity conditions and a counterexample to convergence can be found in 
\cite{GSd}.

We analyze Problems \ref{P1} and \ref{P2} for a radially symmetric potential, $V$. In particular, we find explicit solutions for Problem \ref{P1} and reduce Problem \ref{P2} to an ODE that is the Euler-Lagrange equation of a convex functional when $\alpha=0$.

Usually, we regard the Hamilton-Jacobi equation as an equation for the value function, $u$, and the Fokker-Planck equation as an equation for the density, $m$. Unfortunately, the Hamilton-Jacobi equation is notoriously hard to solve
for the value function and, similarly, the Fokker-Planck equation, as an equation for $m$, also looks somewhat hopeless. 
Our key observation is that, instead, we can regard the Fokker-Planck equation as an equation for the value function, $u$.
A common feature in our analysis of Problems \ref{P1}, \ref{P2}
 is that the Fokker-Planck equations in \eqref{mainfo} and \eqref{mainso} 
 are integrable with respect to $u$.
Moreover, these equations have fairly explicit solutions. Thus, using their solutions in the corresponding Hamilton-Jacobi equation, we transform the MFG in a scalar problem for the density, $m$. This approach is inspired by \cite{GNPr'16}, \cite{GNPr2'16}, \cite{nurbekyan'17}.

\section{Radial solutions for first-order mean-field games}\label{sec:P1'}

Here, we find explicit solutions for Problem \ref{P1}, when $V$ is radially symmetric. For $x\in \Rr^d$, let $r=|x|=(x_1^2+x_2^2+\ldots +x_d^2)^{1/2}$ and assume that $V(x)=V(r),\ x\in B_{1}$. For simplicity, we assume that
\begin{equation*}
g(m)=m^\beta
\end{equation*}
for some $\beta>0$. This previous assumption is not critical, and a similar analysis can be performed for other choices of $g$. The game-theoretical interpretation of MFGs suggests that any solution of Problem \ref{P1}, $(u,m)$, should be radially symmetric. Hence, we assume that
\begin{equation}\label{radialvar}
\begin{cases}
&u(x)=u(r),\ x\in B_{1}\\
&m(x)=m(r),\ x\in B_{1}.
\end{cases}
\end{equation}
Consequently, \eqref{mainfo} takes the form

\begin{equation}\label{mainradialfo}
\begin{cases}
&\frac{u'(r)^2}{2 m(r)^{\alpha}} + V(r)=m(r)^\beta+\Hh,\ \mbox{for}\ r\in\Omega^*\\
&\\
&(1-\alpha) m'(r) u'(r)+m(r)(u''(r)\\
&+\frac{d-1}{r}u'(r))=0,\ \mbox{for}\ r\in \Omega^*\\
&\\
&m(r)>0,\ \mbox{for}\ r\in \Omega^*,\\
&\int\limits_{0}^\infty r^{d-1}m(r)dr=\frac{1}{|\partial B_1|},
\end{cases}
\end{equation}
where
\begin{equation}\label{eq:omega*}
\Omega^*=
\begin{cases}
[0,\infty),\ \mbox{if} \ \Omega=\Rr^d\\
(0,\infty),\ \mbox{if} \ \Omega=\Rr^d \setminus \{0\}.
\end{cases}
\end{equation}

\begin{ART}
	{"FirstOrder3DimensionalRadial",
	MKF := Function[#1, #2] &;
	PDEtoRadial[PDE_, DepVars_, IndVars_] :=
	Module[{pdeop}, 
	pdeop = MKF[DepVars, PDE];
	First@Assuming[r > 0, 
	pdeop @@ (MKF[IndVars, #@Sqrt[Total[#^2 & /@ IndVars]]] & /@ 
	DepVars) /. 
	Solve[Total[#^2 & /@ IndVars] == r^2, First[IndVars]] // 
	Simplify]
	];
	RadialSolutionsPDE[PDE_, DepVars_, IndVars_] :=	
	DSolve[# == 0 & /@ PDEtoRadial[PDE, DepVars, IndVars], # @@ {r} & /@
	DepVars, {r}];
				
	Clear[H, u, x,y,z];			
	PDEtoRadial[
	{(D[u[x, y, z], x]^2 + D[u[x, y, z], y]^2 + 
		D[u[x, y, z], z]^2)/(2 m[x, y, z]^ a) + 
		V[Sqrt[x^2 + y^2 + z^2]] - g[m[x, y, z]] - H, 
		-Div[{D[u[x, y, z], x] m[x, y, z]^(1 - a), 
		D[u[x, y, z], y] m[x, y, z]^(1 - a), 
		D[u[x, y, z], z] m[x, y, z]^(1 - a)}, {x, y, z}]}, {u, m}, {x, y, z}  ]
}
\end{ART}

\begin{pro}\label{pro:reduction}
	Suppose that $\Omega=\Rr^d$ or $\Omega=\Rr^d\setminus\{0\}$. Furthermore, assume that $(u,m,\Hh)\in C^{\infty}(\Omega)\times C^{\infty}(\Omega) \times \Rr$ is a radially symmetric solution for \eqref{mainfo}; that is, \eqref{radialvar} holds. Then, there exists a constant $j \in \Rr$ such that
	\begin{equation}\label{eq:reduction}
	j=u'(r)m(r)^{1-\alpha}r^{d-1}\quad \mbox{for all}\quad r\in \Omega^*.
	\end{equation}
\end{pro}
\begin{remark}
	The right-hand side of \eqref{eq:reduction} is the current of the population across the sphere of radius $r$. Therefore, the previous proposition asserts that the current is constant for all radii $r$. This property of smooth solutions is also valid in one-dimensional MFG models, see \cite{GNPr'16}, \cite{GNPr2'16}, \cite{nurbekyan'17}.
\end{remark}
\begin{proof}[Proof of Proposition \ref{pro:reduction}]
	It suffices to note that
	\begin{align*}
	&\left(m(r)^{1-\alpha}r^{d-1}u'(r)\right)'=m^{-\alpha}r^{d-1}[  (1-\alpha) m'(r) u'(r)\\
	&+m(r)(u''(r)+\frac{d-1}{r}u'(r)) ]=0.
	\end{align*}
\end{proof}
\begin{ART}
	{"ProofPropositionII1",
(*	Clear["Global`*"] *)
{D[m[r]^(1 - \[Alpha]) r^(d - 1) u'[r], r], D[m[r]^(1 - \[Alpha]) r^(d - 1) u'[r], r]/(r^(d-1)m[r]^(-\[Alpha]) )//Simplify}						
		}
\end{ART}
As a corollary of Proposition \ref{pro:reduction}, we characterize smooth solutions for \eqref{P1} in $\Omega=\Rr^d$.
\begin{pro}\label{pro:smoothfo}
	Suppose that $\Omega=\Rr^d$. Furthermore, assume that $(u,m,\Hh)\in C^{\infty}(\Omega)\times C^{\infty}(\Omega) \times \Rr$ is a radially symmetric solution for \eqref{mainradialfo}. Then, necessarily,
	\begin{equation}\label{eq:solsR^d}
	u=\mbox{const},\quad m(r)=\left(V(r)-\Hh\right)^{\frac{1}{\beta}},
	\end{equation}
	where $\Hh$ is such that
	\begin{equation}\label{eq:int=1}
	\int\limits_{0}^\infty r^{d-1}\left(V(r)-\Hh\right)^{\frac{1}{\beta}}dr=\frac{1}{|\partial B_1|}.
	\end{equation}
	Consequently, if for a given $V$ there does not exist $\Hh$ so that \eqref{eq:int=1} holds, there are no smooth solutions for \eqref{mainradialfo}.
\end{pro}
\begin{proof}
	By Proposition \ref{pro:reduction}, we have that $u$ and $m$ satisfy \eqref{eq:reduction}. Furthermore, $u$ is radially symmetric and smooth so $u'(0)=0$. Thus, $j=0$ in \eqref{eq:reduction}. Consequently, $u'(r)=0$ for all $r$ and we obtain \eqref{eq:solsR^d}.
\end{proof}
\begin{example}
	Suppose $d=2,\ \beta=1$ and $V(r)=e^{-\frac{r^2}{2}}$. We claim that in this case, \eqref{P1} does not admit smooth, radially symmetric solutions in $\Omega=\Rr^d$. Indeed, by the previous proposition,  $m$ must have the form \eqref{eq:solsR^d}. Furthermore, the only number $\Hh$ for which the integral in \eqref{eq:int=1} converges is $\Hh=0$. On the other hand, we have
	\[
	\int\limits_{0}^\infty r e^{-\frac{r^2}{2}}dr=1 \neq \frac{1}{2\pi}=\frac{1}{|\partial B_1|}.
	\]
	Therefore, there are no smooth solutions in this case.
\end{example}

The previous example and Proposition \ref{pro:smoothfo} assert that the class of smooth, radially symmetric solutions for \eqref{mainfo} is restricted to the solutions of the form \eqref{eq:solsR^d}. This restriction is not surprising taking into account game-theoretical interpretation of MFGs. Indeed, agents seek to maximize $V$ (to be at a better location) at the lowest possible cost (modeled by the Lagrangian). Therefore, if the spatial preference depends only on the distance from a given center (origin here) agents choose to move only in the radial direction to avoid extra cost of moving in lateral directions. Moreover, the value function, $u$, is differentiable at $x\in \Omega$ if at location $x$ there is a unique optimal strategy. Therefore, if $\Omega=\Rr^d$ then $u$ is smooth at 0 if and only if the optimal strategy for agents at the origin is to stand still. Furthermore, the Fokker-Planck equation in \eqref{mainfo} states that the total mass of the agents in a given area does change as a result of the actions of the agents. Hence, the current of agents in two concentric spheres with two different radii is the same - otherwise, there will be an accumulation or leakage of agents in the spherical strip between these two spheres. Thus, if the current is 0 at the origin, it is 0 everywhere. Hence, the only possibility for smooth radial solutions in $\Rr^d$ (or a ball centered at 0) is when $u$ is constant; that is, agents do not move.

As we describe below, the class of smooth solutions for $\Omega=\Rr^d \setminus \{0\}$ is much richer. The reason is that the current which is still constant in this case does not have to be 0. One may think of this as agents coming from the origin and spreading to $\infty$ or vice versa.

Radially symmetric MFGs are relevant in large-population models where the spatial preferences of agents depend only on the distance from a given location. For instance, in an evenly developed city, the desirability of a house frequently depends on its distance from the downtown or city center. It is particularly interesting whether stationary solutions found here are stable in a sense given earlier.

Next, we study radially symmetric solutions for \eqref{P1} in $\Omega=\Rr^d \setminus\{0\}$. For that, 
let
\begin{equation}\label{eq:Fj}
F_j(t)=\frac{j^2}{2}t^{\frac{\alpha-2}{\beta}}-t,\quad t>0.
\end{equation}

Note that $F_j$ is a decreasing function with $F_j(\Rr^+)=\Rr$.
\begin{pro}
	Suppose that $\Omega=\Rr^d \setminus \{0\}$ and that $V$ is such that 
	\begin{equation}\label{eq:Vdecay}
	\lim\limits_{r\to 0}V(r)r^{\frac{2\beta(d-1)}{2+\beta-\alpha}}=\lim\limits_{r\to \infty}V(r)r^{\frac{2\beta(d-1)}{2+\beta-\alpha}}=0.
	\end{equation}
	Also, assume that
	\begin{equation}\label{eq:alphabounds}
	\frac{2}{d}<\alpha<\min\{2,\frac{2}{d}+\beta\}.
	\end{equation}
	Furthermore, let $(u,m,\Hh)\in C^{\infty}(\Omega)\times C^{\infty}(\Omega) \times \Rr$ be a radially symmetric solution for \eqref{mainradialfo}. Then, either $(u,m,\Hh)$ are given by \eqref{eq:solsR^d} and \eqref{eq:int=1} or
	\begin{align}\label{eq:solsR^d-0}
	\begin{split}
	m(r)&=r^{-\frac{2(d-1)}{2+\beta-\alpha}}\left[F_j^{-1}\left((\Hh-V(r))r^{\frac{2\beta(d-1)}{2+\beta-\alpha}}\right)\right]^{\frac{1}{\beta}},\\ u(r)&=j\int\limits_1^r m(s)^{1-\alpha} s^{1-d} ds+c,\quad r>0
	\end{split}
	\end{align}
	where $j\neq 0$ and $c$ are arbitrary constants, and $\Hh$, which is unique for a given $j$, is such that

	\begin{align}\label{eq:int=1-0}
	\begin{split}
	\int\limits_{0}^\infty &r^{\frac{(d-1)(\beta-\alpha)}{2+\beta-\alpha}}\left[F_j^{-1}\left((\Hh-V(r))r^{\frac{2\beta(d-1)}{2+\beta-\alpha}}\right)\right]^{\frac{1}{\beta}}dr\\
	&=\frac{1}{|\partial B_1|}.
	\end{split}
	\end{align}
\end{pro}
\begin{proof}
	By Proposition \ref{pro:reduction}, we have that $u$ and $m$ satisfy \eqref{eq:reduction}. If $j=0$, we obtain \eqref{eq:solsR^d} and \eqref{eq:int=1}. Suppose $j\neq 0$. From \eqref{eq:reduction} we get
	\begin{equation*}
	u'(r)=j\frac{m(r)^{\alpha-1}}{r^{d-1}},
	\end{equation*}

	Substituting the expression of $u'$ in the first equation of \eqref{mainradialfo}, we obtain
	\begin{equation}\label{singeqvfo}
	\frac{j^2}{2}r^{2(1-d)}m(r)^{\alpha-2}-m(r)^\beta=\Hh-V(r).
	\end{equation}
	Let
	\begin{equation}\label{eq:sigma}
	\sigma=\frac{2\beta(d-1)}{2+\beta-\alpha}.
	\end{equation}
	Then, multiplying \eqref{singeqvfo} by $r^\sigma$, we get
	\begin{equation*}
	F_j\left(r^\sigma m(r)^\beta\right)=(\Hh-V(r))r^\sigma.
	\end{equation*}
	\begin{ART}
		{"ComputationEpxressionII11  ",
			(* Clear["Global`*"] *)
			Clear[F];
			sigma = (2 \[Beta] (d - 1))/(2 + \[Beta] - \[Alpha]);
			F = j^2/2 #^((\[Alpha] - 2)/\[Beta]) - # &;
			qq=PowerExpand[F[r^sigma m[r]^\[Beta]]];
			Clear[F];
			qq
			(* Clear[\[Sigma]] *)
			}
	\end{ART}	
	Thus, we obtain \eqref{eq:solsR^d-0}.
	
	Now, we show that there exists a unique $\Hh$ such that \eqref{eq:int=1-0} holds.
	For $\Hh \in \Rr$, let 
	\begin{equation}\label{eq:phiH}
	\begin{split}
	&\phi (\Hh)=\\
	& \int\limits_{0}^\infty r^{\frac{(d-1)(\beta-\alpha)}{2+\beta-\alpha}}\left[F_j^{-1}\left((\Hh-V(r))r^{\frac{2\beta(d-1)}{2+\beta-\alpha}}\right)\right]^{\frac{1}{\beta}}dr.
	\end{split}
	\end{equation}
	\begin{ART}
	{"rExponentIntegrandExpressionII13",
		ss = (2 (d - 1))/(2 + \[Beta] - \[Alpha]);
		-ss + d - 1 // Simplify
		}
	\end{ART}
	It is straightforward to check that if \eqref{eq:Vdecay} and \eqref{eq:alphabounds} hold, then $\phi(\Hh)<\infty$ if and only if $\Hh> 0$. Moreover, since $F_j$ is decreasing, $F_j^{-1}$ is also decreasing. Therefore, by the Monotone Convergence Theorem,  $\phi(\Hh)$ is decreasing, continuous, and 
	\begin{align*}
	&\lim\limits_{\Hh\to 0+}\phi(\Hh)=+\infty,\\
	&\lim\limits_{\Hh\to \infty}\phi(\Hh)=\\
	&\int\limits_{0}^\infty\lim\limits_{\Hh\to \infty} r^{\frac{(d-1)(\beta-\alpha)}{2+\beta-\alpha}}\left[F_j^{-1}\left((\Hh-V(r))r^{\frac{2\beta(d-1)}{2+\beta-\alpha}}\right)\right]^{\frac{1}{\beta}}dr\\
	&=0.
	\end{align*}
	Thus, there exists a unique $\Hh$ such that $\phi(\Hh)=\frac{1}{|\partial B_1|}$.
\end{proof}

The larger is the potential at a reference location the more desirable is this location. Therefore, if $V$ is not small enough at $\infty$ agents may be attracted to $\infty$ too much, and we may end up having infinite mass. Similarly, if $V$ is too large near the origin there may be too many agents near the origin, and this may also yield infinite mass. Hence, $V$ must be small enough at the infinity and the origin. These considerations are quantified in \eqref{eq:Vdecay}.

\subsection{Numerical solutions}
Here, we numerically solve Problem \eqref{P1} using \eqref{mainradialfo}.  
We choose $d=2$, the coupling $g(m) = m$ and current $j=1$. We perform two experiments with the congestion exponent $\alpha$ varying in the set $\alpha\in\{ 1.3, 1.4, 1.5, 1.6\}$. 

First, we choose the case with potential $V(x)=e^{-\frac{|x|^2}{2}}\sin\left(\pi(|x|+\frac 1 4)\right)$ as in Figure \ref{fig:plotV}. For this case, in Figure \ref{fig:plotPhi}, we illustrate the behavior of $\phi$ given by \eqref{eq:phiH}, where the integral computed over the interval $[0,100]$.
For this example, for each $\alpha\in\{ 1.3, 1.4, 1.5, 1.6\}$, the value of $\Hh$ for which $\phi(\Hh)=\frac{1}{|\partial B_1|}$, is approximately $\{13.48, 15.99, 25.05, 62.26\}$, respectively.
Moreover, the value function $u$ and the density $m$ are illustrated in \ref{fig:plotu} and \ref{fig:plotm}, respectively, for each $\alpha$ and corresponding $\Hh$.

Next, we choose the potential $V(x)=(1+|x|)^{-\frac 3 2}\sin\left(2\pi(|x|+\frac{1}{4})\right)$ (see Figure \ref{fig:plotVPowerSin}).
We choose the same congestion exponents, $\alpha$. The corresponding $\Hh$'s for this case are approximately $\{13.47, 15.95, 25, 62.20\}$, respectively.

\begin{IMG}
	Clear[F];
	Get[CloudObject["WorkSpaceExpVSin5000Corrected"]];
\end{IMG}

	\begin{IMG}
	{"plotPhi", 
	GraphicsRow[{Plot\[Phi]}, ImageSize -> Large]
	}
	\end{IMG}

	\begin{IMG}
	{"plotu", 
	GraphicsRow[{Plotu}, ImageSize -> Large]
	}
	\end{IMG}

	\begin{IMG}
	{"plotV", 
	GraphicsRow[{PlotV}, ImageSize -> Large]
	}
	\end{IMG}	

	\begin{IMG}
	{"plotm", 
	GraphicsRow[{Plotm}, ImageSize -> Large]
	}
	\end{IMG}	
		
	\begin{IMG}
	{"clear",
	Clear[PlotmV,Plotu,Plot\[Phi],u,m,x,H]
		}
	\end{IMG}

\begin{figure}[htb]
\centering
\begin{subfigure}[b]{\sizefigure\textwidth}
\includegraphics[width=\textwidth]{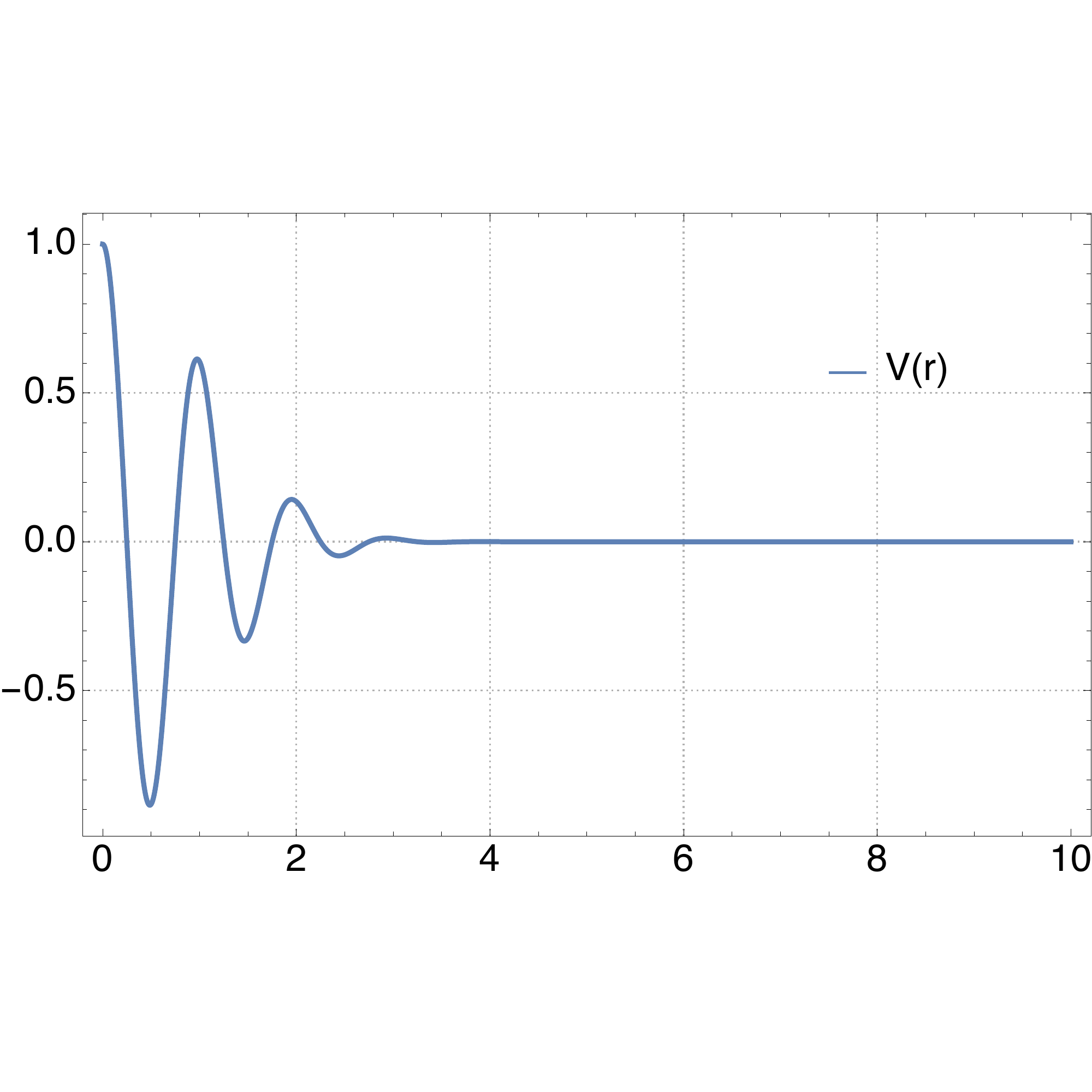}
\caption{The potential $V$}
\label{fig:plotV}
\end{subfigure}
\begin{subfigure}[b]{\sizefigure\textwidth}
\includegraphics[width=\textwidth]{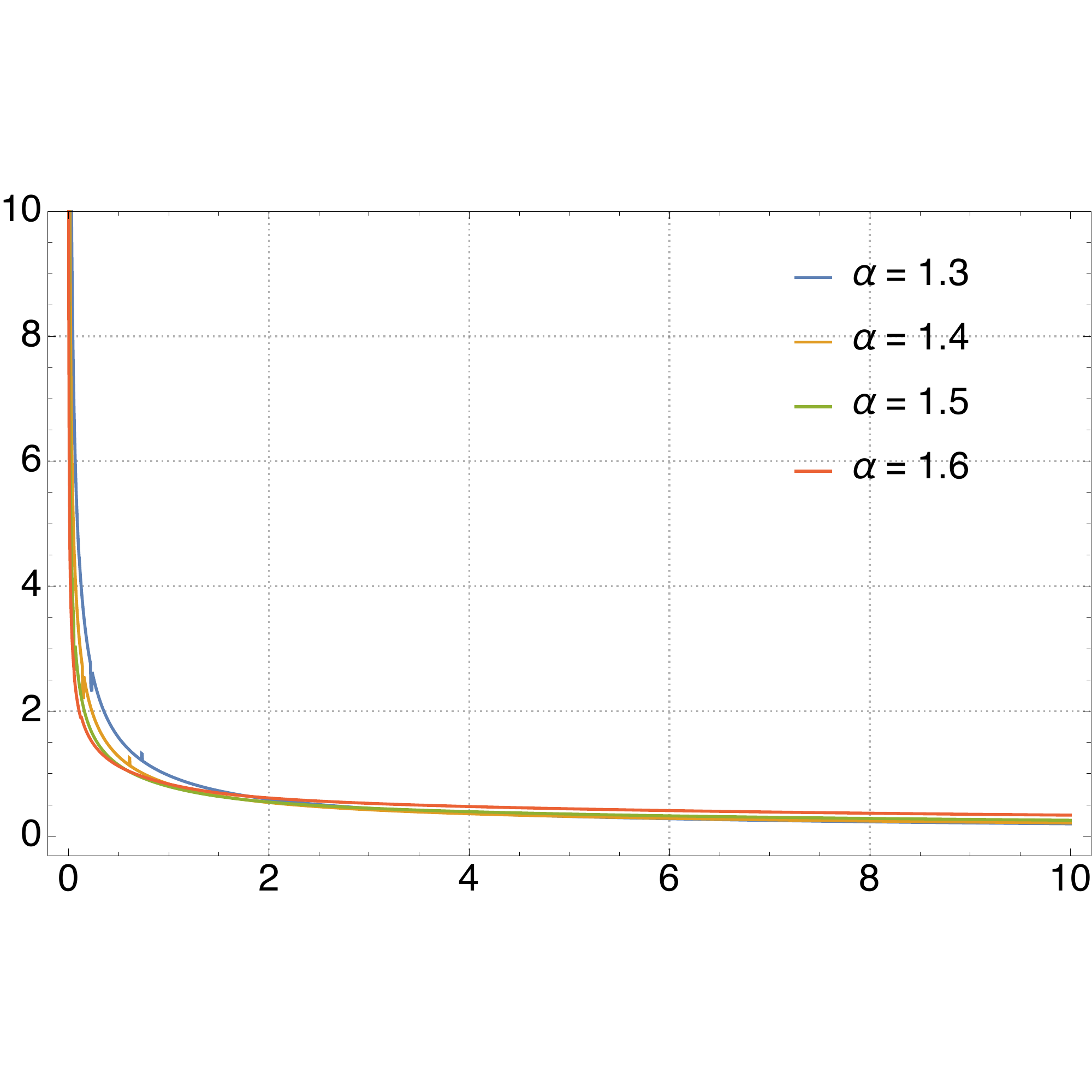}
\caption{$\phi$ defined in \eqref{eq:phiH}}
\label{fig:plotPhi}
\end{subfigure}
\begin{subfigure}[b]{\sizefigure\textwidth}
\includegraphics[width=\textwidth]{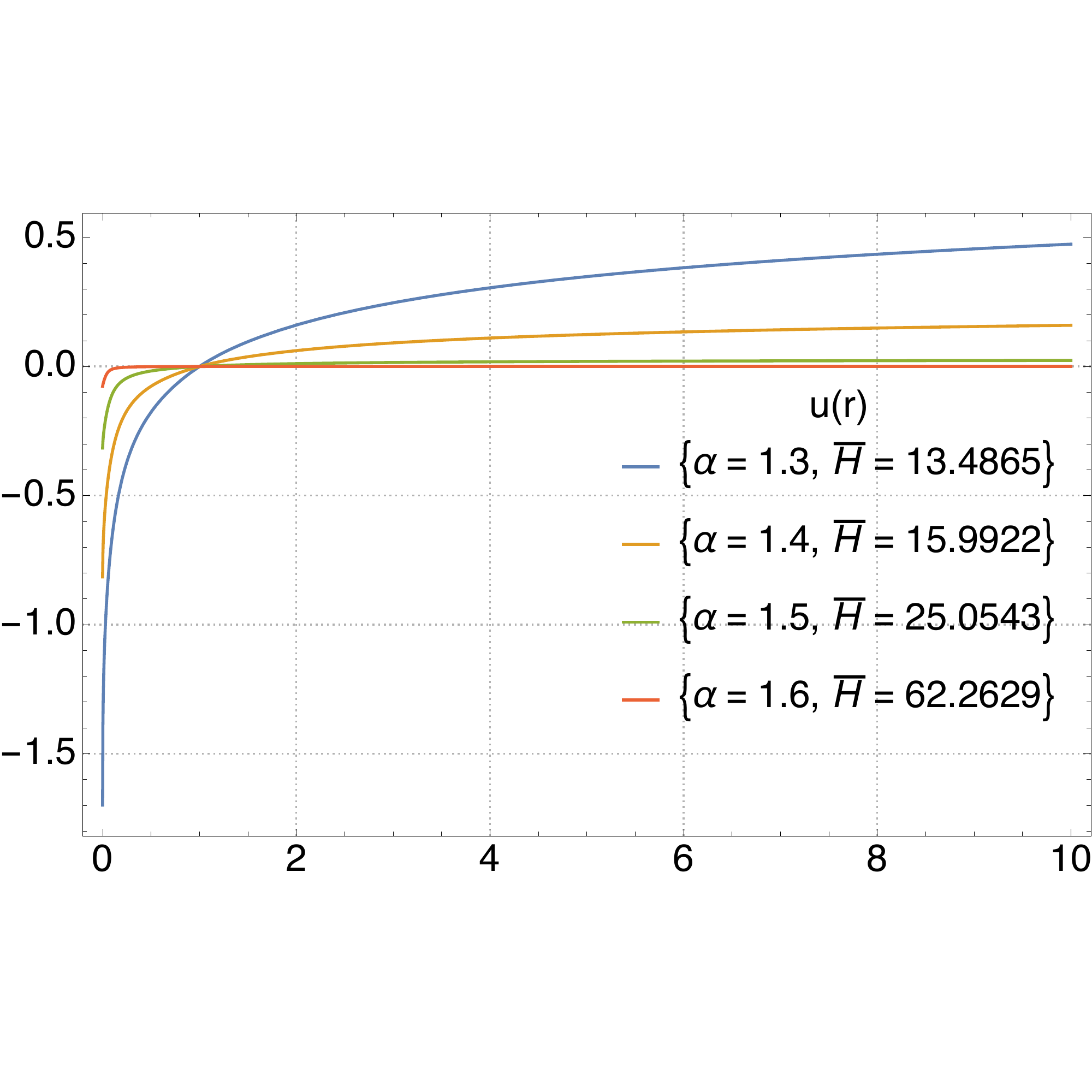}
\caption{The value function $u$}
\label{fig:plotu}
\end{subfigure}
\begin{subfigure}[b]{\sizefigure\textwidth}
\includegraphics[width=\textwidth]{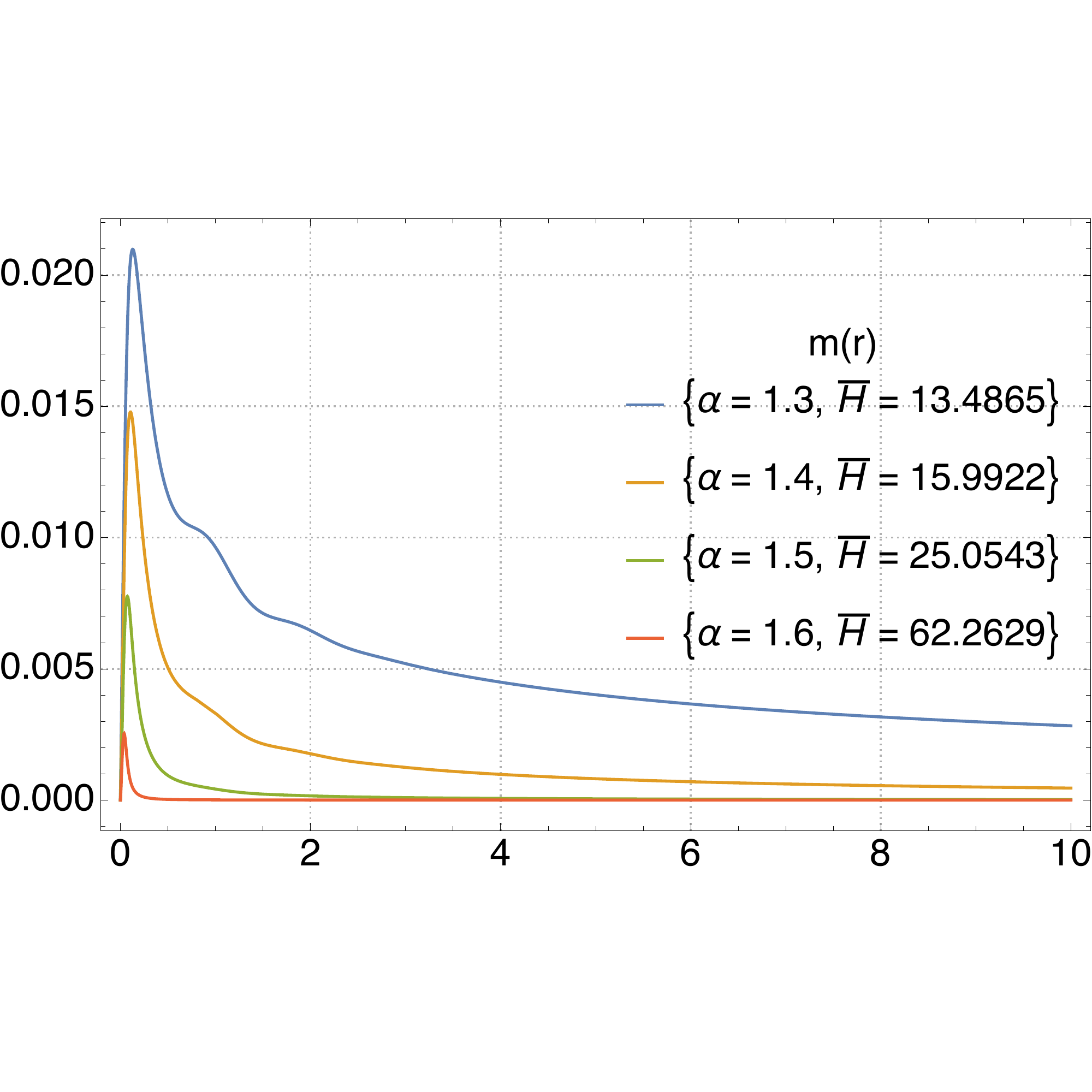}
\caption{The density $m$}
\label{fig:plotm}
\end{subfigure}
~ 
\caption{Numerical solution of Problem \ref{P1} for $d=2$, $g(m) = m$, $j=1$, $V(x)= e^{-\frac{|x|^2}{2}} \sin\left(2\pi(|x|+\frac{1}{4})\right)$ and  $\alpha\in\{ 1.3, 1.4, 1.5, 1.6\}$.}\label{fig:solFirstOrderMFG}
\end{figure}

\begin{IMG}
Clear[F];
Get[CloudObject["WorkSpacePowerVSin"]];
\end{IMG}

\begin{IMG}
{"plotPhiPowerSin", 
	GraphicsRow[{Plot\[Phi]}, ImageSize -> Large]
}
\end{IMG}

\begin{IMG}
{"plotuPowerSin", 
	GraphicsRow[{Plotu}, ImageSize -> Large]
}
\end{IMG}

\begin{IMG}
{"plotVPowerSin", 
	GraphicsRow[{PlotV}, ImageSize -> Large]
}
\end{IMG}	

\begin{IMG}
{"plotmPowerSin", 
	GraphicsRow[{Plotm}, ImageSize -> Large]
}
\end{IMG}	

\begin{IMG}
{"clear",
	Clear[PlotmV,Plotu,Plot\[Phi],u,m,x,H]
}
\end{IMG}

\begin{figure}[htb]
\centering
\begin{subfigure}[b]{\sizefigure\textwidth}
\includegraphics[width=\textwidth]{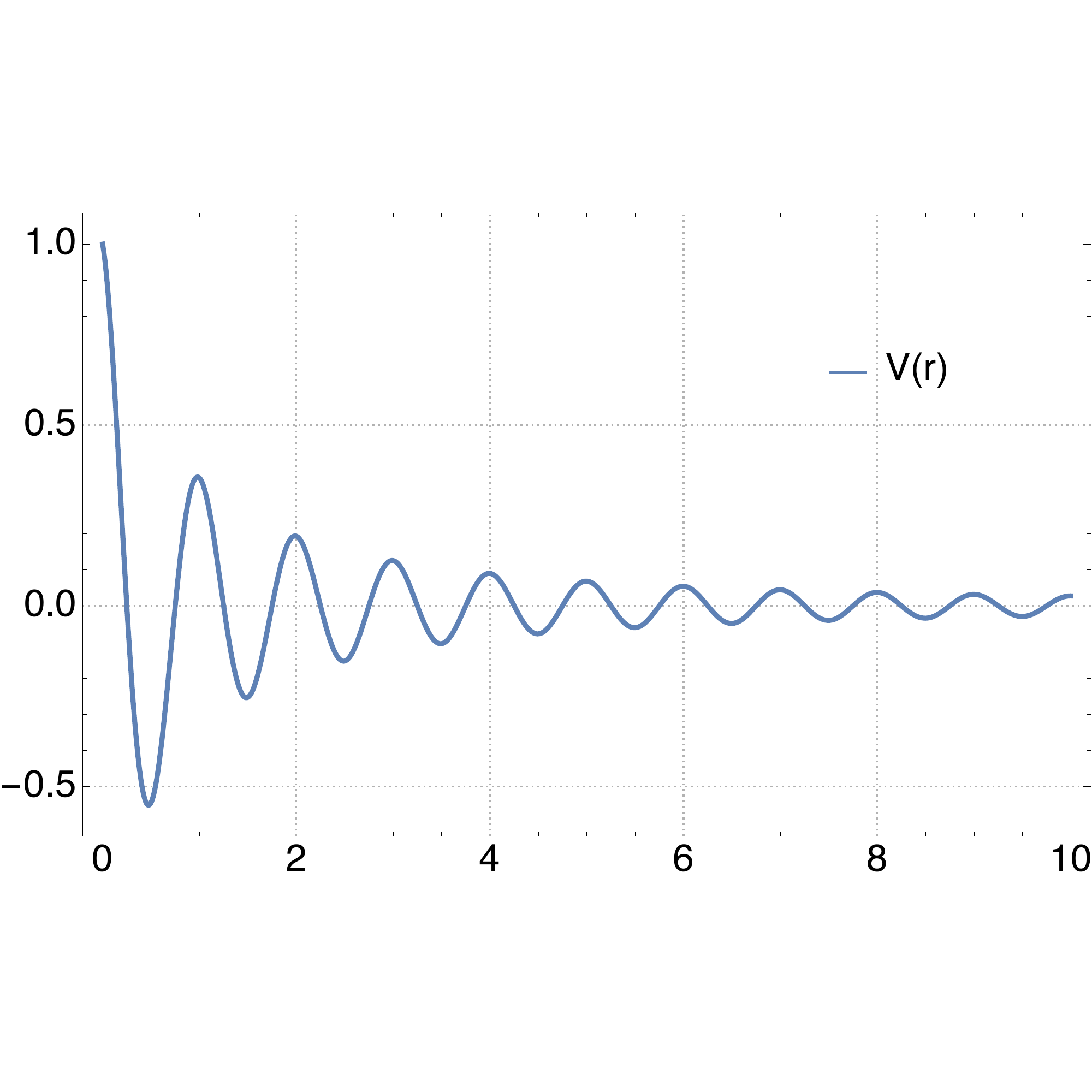}
\caption{The potential $V$}
\label{fig:plotVPowerSin}
\end{subfigure}
\begin{subfigure}[b]{\sizefigure\textwidth}
\includegraphics[width=\textwidth]{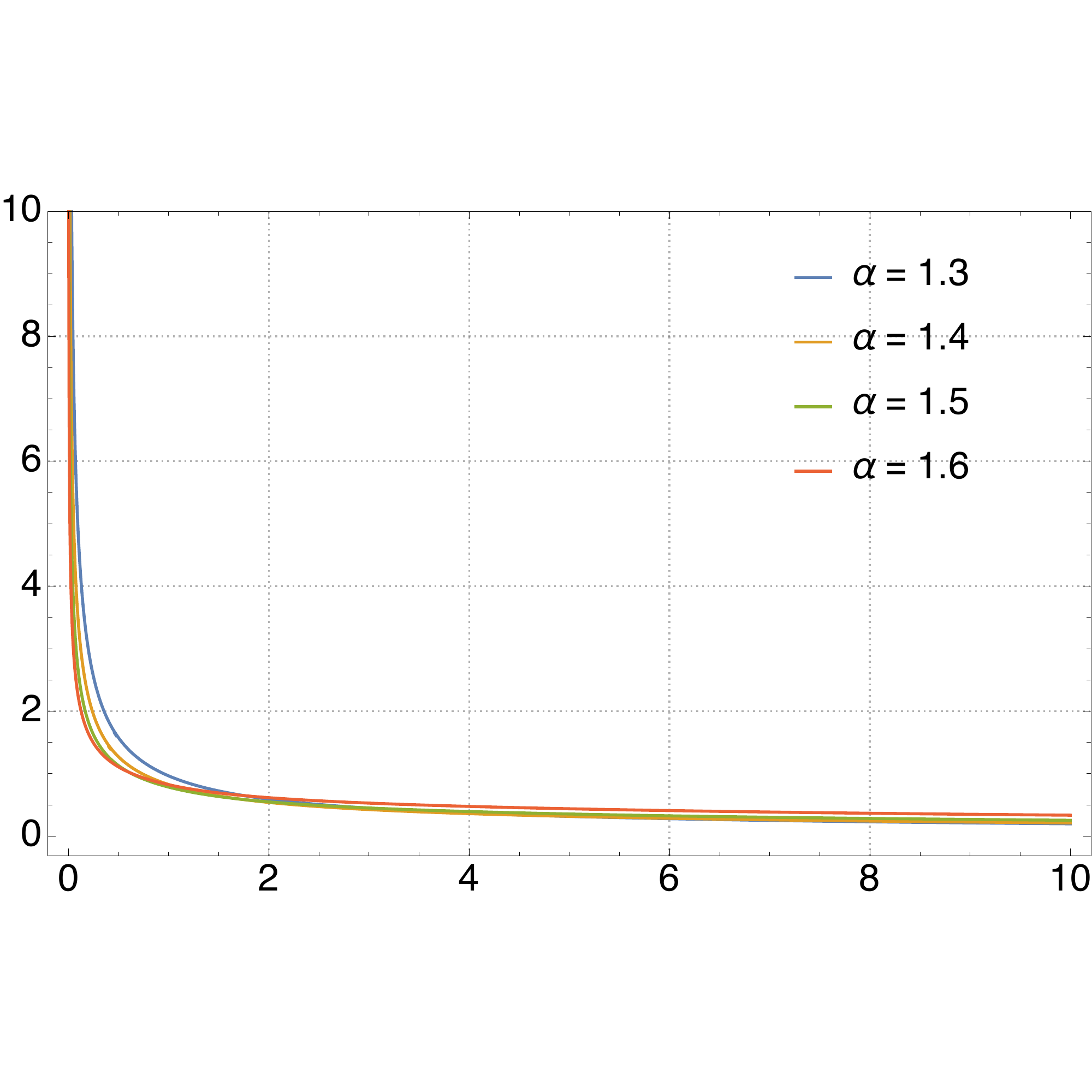}
\caption{$\phi$ defined in \eqref{eq:phiH}}
\label{fig:plotPhiPowerSin}
\end{subfigure}
\begin{subfigure}[b]{\sizefigure\textwidth}
\includegraphics[width=\textwidth]{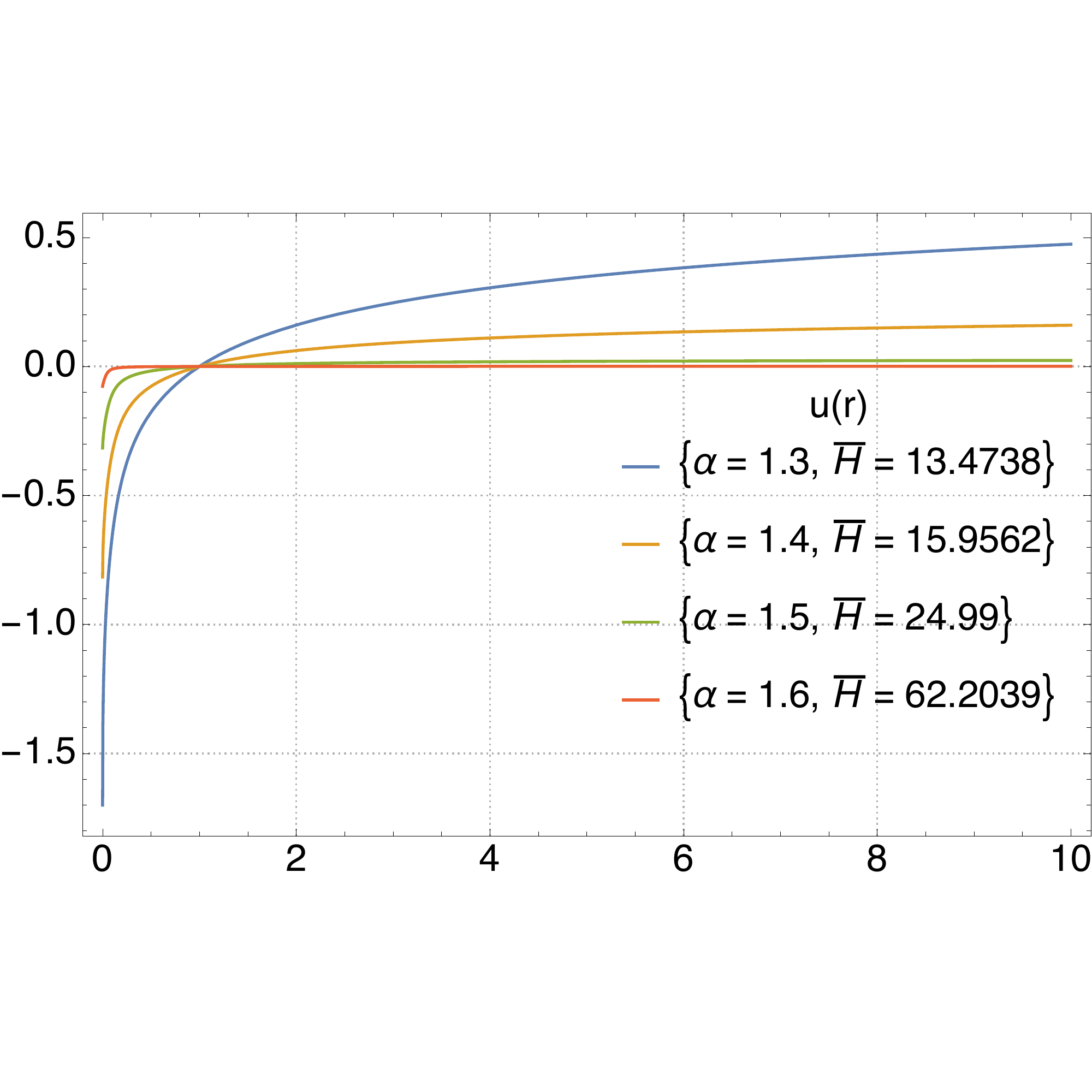}
\caption{The value function $u$}
\label{fig:plotuPowerSin}
\end{subfigure}
\begin{subfigure}[b]{\sizefigure\textwidth}
\includegraphics[width=\textwidth]{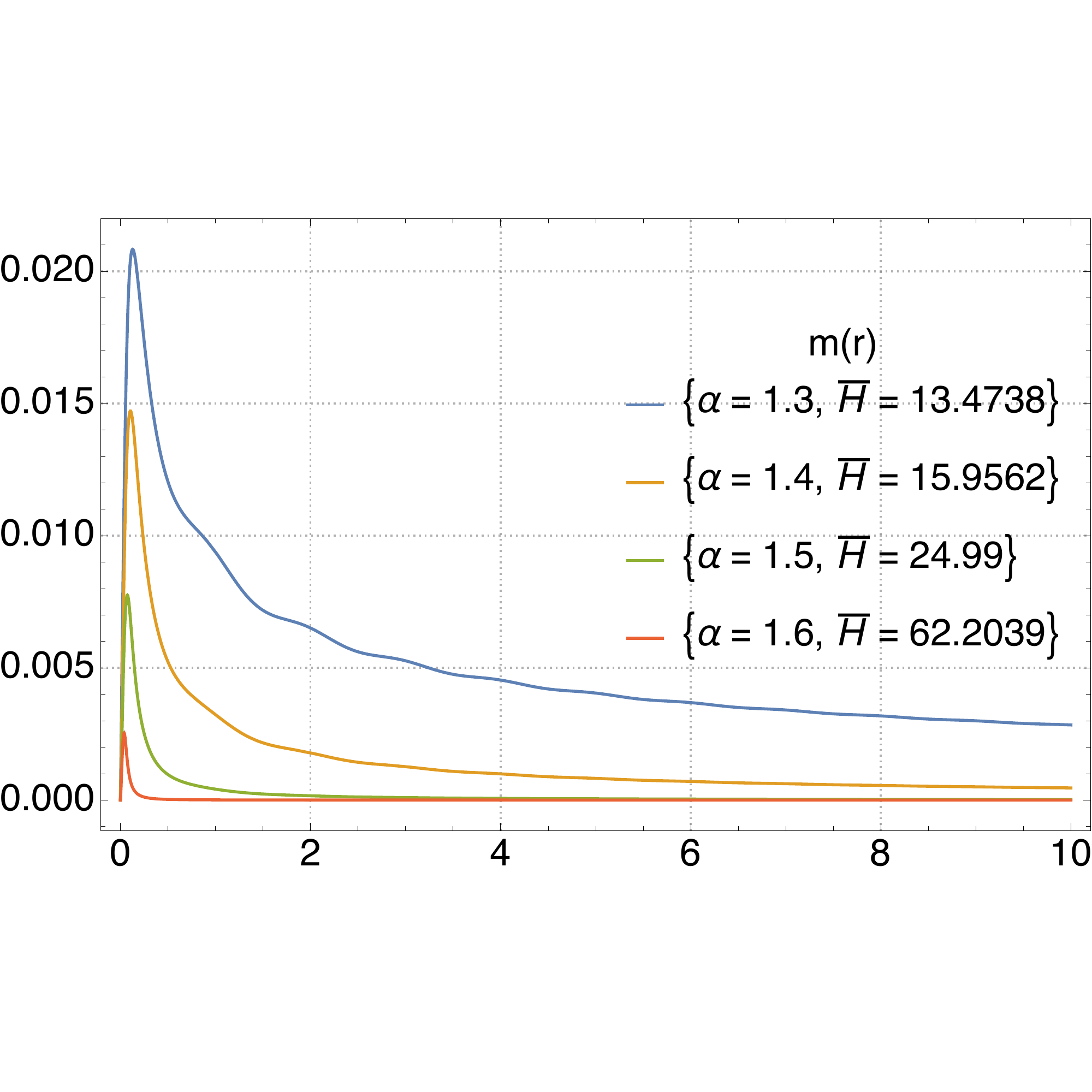}
\caption{The density $m$}
\label{fig:plotmPowerSin}
\end{subfigure}
~ 
\caption{Numerical solution of Problem \ref{P1} for $d=2$, $g(m) = m$, $j=1$, $V(x)=(1+|x|)^{-\frac 3 2}\sin(\pi \left(|x|+\frac{1}{4})\right) $ and  $\alpha\in\{1.3, 1.4, 1.5, 1.6\}$.}\label{fig:solFirstOrderMFGPowerSin}
\end{figure}

In Figure \ref{fig:solFirstOrderMFG}, we see that for a fast decaying potential and a non-zero radial current, we have localized solutions near the origin. The radial decay of the solution is expected because the total mass is $1$ and the agents are moving away from the origin. A similar behavior can be observed in Figure \ref{fig:solFirstOrderMFGPowerSin} for a different potential. We also see that the oscillations in the potential do not affect substantially the solution. Finally, we depict a 2D solution 
in Figure \ref{fig:solFirstOrderMFGPowerSinR}.

\section{Second-order case}

Here, we study radially symmetric solutions of Problem \ref{P2}. As before, we assume that $u,m$ are radial as in \eqref{radialvar}. Accordingly, Problem \ref{P2} takes the form
\begin{align}\label{mainradialso}
\begin{cases}
&-u''(r)-\frac{d-1}{r}u'(r)+\frac{u'(r)^2}{2 m(r)^{\alpha}}+V(r)\\
&\qquad =g(m(r))+\Hh\ \mbox{for}\ r\in\Omega^*\\
&\\
&-m''(r)-\frac{d-1}{r}m'(r)\\
&\quad -m(r)^{1-\alpha}\left(u''(r)+\frac{d-1}{r}u'(r)\right)\\
&\quad -(1-\alpha)m(r)^{-\alpha}u'(r)m'(r)=0 \ \mbox{for}\ r\in\Omega^*\\
&\\
&m(r)>0,\ \mbox{for}\ r\in \Omega^*,\\
& \int\limits_{0}^\infty r^{d-1}m(r)dr=\frac{1}{|\partial B_1|}.
\end{cases}
\end{align}
\begin{ART}
	{"radial coordinates second order in two dimensions", 
MKF := Function[#1, #2] &;
PDEtoRadial[PDE_, DepVars_, IndVars_] :=
Module[{pdeop}, 
pdeop = MKF[DepVars, PDE];
First@Assuming[r > 0, 
pdeop @@ (MKF[IndVars, #@Sqrt[Total[#^2 & /@ IndVars]]] & /@ 
DepVars) /. 
Solve[Total[#^2 & /@ IndVars] == r^2, First[IndVars]] // 
Simplify]
];
RadialSolutionsPDE[PDE_, DepVars_, IndVars_] :=
DSolve[# == 0 & /@ PDEtoRadial[PDE, DepVars, IndVars], # @@ {r} & /@
DepVars, {r}];
PDEtoRadial[
{-Laplacian[
	u[x, y], {x, y}] + (D[u[x, y], x]^2 + 
	D[u[x, y], y]^2)/(2 m[x, y]^\[Alpha]) + V[Sqrt[x^2 + y^2]] - 
	g[ m[x, y]] - H, -Laplacian[m[x, y], {x, y}] -
	Div[{D[u[x, y], x] m[x, y]^(1 - \[Alpha]), 
		D[u[x, y], y] m[x, y]^(1 - \[Alpha])}, {x, y}]}, {u, m}, {x, 
	y}] // Expand
	}
\end{ART}

As in the first-order case, the Fokker-Planck equation in \eqref{mainradialso} is integrable.
\begin{pro}\label{pro:reductionelliptic}
	Suppose that $\Omega=\Rr^d$ or $\Omega=\Rr^d\setminus\{0\}$. Furthermore, assume that $(u,m,\Hh)\in C^{\infty}(\Omega)\times C^{\infty}(\Omega) \times \Rr$ is a radially symmetric solution for \eqref{mainso}; that is, \eqref{radialvar} holds. Then, there exists a constant $j \in \Rr$ such that
	\begin{equation}\label{eq:reductionelliptic}
	j=u'(r)m(r)^{1-\alpha}r^{d-1}+m'(r)r^{d-1}\quad \mbox{for all}\quad r\in \Omega^*.
	\end{equation}
\end{pro}
\begin{proof}
	We just note that
	\begin{align*}
&\left(u'(r)m(r)^{1-\alpha}r^{d-1}+m'(r)r^{d-1}\right)'\\
&\quad =r^{d-1}\Big[m''(r)+\frac{d-1}{r}m'(r)\\
&\quad +m(r)^{1-\alpha}\left(u''(r)+\frac{d-1}{r}u'(r)\right)\\
&\quad +(1-\alpha)m(r)^{-\alpha}u'(r)m'(r)\Big]=0.
	\end{align*}
\begin{ART}	
{"we check the identity in the proof of proposition pro:reductionelliptic - zero is the answer!",
-r^(d - 1) (-(((-1 + d) Derivative[1][m][r])/
r) - ((-1 + d) m[r]^(1 - \[Alpha]) Derivative[1][u][r])/r - 
m[r]^-\[Alpha] Derivative[1][m][r] Derivative[1][u][
r] + \[Alpha] m[r]^-\[Alpha] Derivative[1][m][
r] Derivative[1][u][r] - m''[r] - m[r]^(1 - \[Alpha]) u''[r]) -
D[u'[r] m[r]^(1 - \[Alpha]) r^(d - 1) + m'[r] r^(d - 1), 
r] // FullSimplify
	}	
\end{ART}	
\end{proof}
Using the previous proposition, we reduce \eqref{mainradialso} to a single second-order ODE, which is linear in the highest order derivative.
\begin{proposition}\label{pro:ODErho}
	Suppose that $\Omega=\Rr^d$ or $\Omega=\Rr^d\setminus\{0\}$. Furthermore, assume that $(u,m,\Hh)\in C^{\infty}(\Omega)\times C^{\infty}(\Omega) \times \Rr$ is a radially symmetric solution for \eqref{mainso}. Then, the function $r\mapsto\rho(r)=m(r)^{\alpha+\frac{1}{2}},\ r\in \Omega^*$ (see \ref{eq:omega*} for $\Omega^*$) satisfies the following ODE:
	\begin{align}\label{eq:ODErhoj}
	\begin{split}
	&\rho''(r)+\rho'(r)\left(\frac{d-1}{r}-\frac{\alpha j r^{1-d}}{\rho(r)^{\frac{2}{2\alpha+1}}}\right)\\
	=&\left(\alpha+\frac{1}{2}\right)\left(g\left(\rho(r)^{\frac{2}{2\alpha+1}}\right)+\Hh-V(r)\right)\rho(r)^{\frac{1}{2\alpha+1}}\\
	&-\frac{(2\alpha+1)j^2r^{2-2d}\rho(r)^{\frac{2\alpha-3}{2\alpha+1}}}{4},\quad r\in \Omega^*.
	\end{split}
	\end{align}
	Furthermore, if $\alpha=0$ the previous ODE takes the form
	\begin{align}\label{eq:ODErhoa=0}
	\begin{split}
	&\rho''(r)+\rho'(r)\frac{d-1}{r}=\frac{1}{2}\left(g\left(\rho(r)^2\right)\right.\\
	&\left.\hspace{1cm}+\Hh-V(r)\right)\rho(r)-\frac{j^2r^{2-2d}}{4\rho(r)^3},\quad r\in \Omega^*.
	\end{split}
	\end{align}
	Moreover, \eqref{eq:ODErhoa=0} is the Euler-Lagrange equation of the optimization problem
	\begin{align}\label{eq:vara=0}
	\begin{split}
	&\inf\limits_{\rho} \int\limits_{\Omega^*} r^{d-1} \left(\rho'(r)^2+\frac{G(\rho(r)^2)}{2}\right.\\
	&\left.-\frac{V(r)\rho(r)^2}{2}+\frac{j^2 r^{2-2d}}{4\rho(r)^2}\right)dr,
		\end{split}
		\end{align}
		where $G$ is the antiderivative of $g$, and $j\in \Rr$ is the constant from Proposition \ref{pro:reduction}, 	
	under the constraint
	\[
	\int\limits_{\Omega^*} \rho(r)^2 r^{d-1}dr=\frac{1}{|\partial B_1|}.
	\]
 Additionally, \eqref{eq:vara=0} can be written in terms of $m$ as
	\begin{align}\label{eq:vara=0m}
	\begin{split}
	&\inf\limits_{m} \int\limits_{\Omega^*} r^{d-1} \left(\frac{m'(r)^2}{4m(r)}+\frac{G(m(r))}{2}\right.\\
	&\left.-\frac{V(r)m(r)}{2}+\frac{j^2 r^{2-2d}}{4 m(r)}\right)dr\\
	&\int\limits_{\Omega^*} m(r) r^{d-1}dr=\frac{1}{|\partial B_1|}.
	\end{split}
	\end{align}
	If $j=0$, \eqref{eq:ODErhoj} takes the form
	\begin{align}\label{eq:ODErhoj=0}
	\begin{split}
	&\rho''(r)+\rho'(r)\frac{d-1}{r}=\left(\alpha+\frac{1}{2}\right)\Big[g\left(\rho(r)^{\frac{2}{2\alpha+1}}\right)\\
	&\hspace{3cm}+\Hh-V(r)\Big]\rho(r)^{\frac{1}{2\alpha+1}}.
	\end{split}
	\end{align}
	Furthermore, \eqref{eq:ODErhoj=0} is the Euler-Lagrange equation of the optimization problem
	\begin{align}\label{eq:varj=0}
	\begin{split}
	&\inf\limits_{\rho} \int\limits_{\Omega^*} r^{d-1} [\rho'(r)^2+(2\alpha+1)G_1(\rho(r))\\
	&\hspace{1cm}+\frac{(2\alpha+1)^2}{2(\alpha+1)}\rho(r)^\frac{2\alpha+2}{2\alpha+1}(\Hh-V(r))]dr
	\end{split}
	\end{align}
	where $\rho\mapsto G_1(\rho)$ is the antiderivative of the map $\rho \mapsto g\left(\rho^{\frac{2}{2\alpha+1}}\right)\rho^{\frac{1}{2\alpha+1}}$.
\end{proposition}
\begin{proof}
	From Proposition \ref{pro:reductionelliptic}, we have that $u$ and $m$ satisfy \eqref{eq:reductionelliptic}. Therefore, we obtain
	\[u'(r)=j (r^{1-d}-m'(r))m(r)^{\alpha-1},\quad r\in \Omega^*.
	\]
\begin{ART}
{"Solving u' in terms of j", 
Solve[u'[r] m[r]^(1 - \[Alpha]) r^(d - 1) + m'[r] r^(d - 1) == j, 
u'[r]]	
}	
\end{ART}			
	We plug-in this expression in the first equation of \eqref{mainradialso} and after elementary manipulations obtain
	\begin{align}\label{eq:ODEm}
	\begin{split}
	&\frac{m''(r)}{m(r)}+\left(\alpha-\frac{1}{2}\right)\frac{m'(r)^2}{m(r)^2}\\
	&+\frac{m'(r)}{m(r)}\left(\frac{d-1}{r}-\frac{\alpha jr^{1-d}}{m(r)}\right)+\frac{j^2 r^{2-2d}}{2m(r)^2}\\ 
	=&\left(g(m(r))+\Hh-V(r)\right)m(r)^{-\alpha},\quad r\in \Omega^*.
	\end{split}
	\end{align}
\begin{ART}
{"replacing u' and u'' in the Hamilton-Jacobi equation", 
-H - g[m[r]] + V[r] - ((d - 1) Derivative[1][u][r])/r + 
1/2 m[r]^-\[Alpha] Derivative[1][u][r]^2 - (u^\[Prime]\[Prime])[
r] /. Derivative[1][u][
r] -> -r^-d m[
r]^(-1 + \[Alpha]) (-j r + r^d Derivative[1][m][r]) /. 
Derivative[1][u]'[r] -> 
D[-r^-d m[r]^(-1 + \[Alpha]) (-j r + r^d Derivative[1][m][r]), 
r] // Simplify		
	}	
\end{ART}	
	Next, for any $\sigma\in \Rr\setminus\{0\}$ we have that
	\begin{align*}
	&\frac{(m(r)^\sigma)'}{\sigma m(r)^\sigma}=\frac{m'(r)}{m(r)}\\
	&\frac{(m(r)^\sigma)''}{\sigma m(r)^\sigma}=\frac{m''(r)}{m(r)}+(\sigma-1)\frac{m'(r)^2}{m(r)^2}.
	\end{align*}
	Therefore, if we choose $\sigma=\alpha+\frac{1}{2}$ and denote by $\rho(r)=m(r)^{\alpha+\frac{1}{2}}$ we obtain \eqref{eq:ODErhoj} from \eqref{eq:ODEm}.
\begin{ART}
{"transforming the hamilton-jacobi equation using \rho", 
Function[m, -H - g[m[r]] + 1/2 j^2 r^(2 - 2 d) m[r]^(-2 + \[Alpha]) + 
V[r] - j r^(1 - d) \[Alpha] m[r]^(-2 + \[Alpha]) Derivative[1][m][
r] - (m[r]^(-1 + \[Alpha]) Derivative[1][m][r])/
r + (d m[r]^(-1 + \[Alpha]) Derivative[1][m][r])/r - 
1/2 m[r]^(-2 + \[Alpha]) Derivative[1][m][r]^2 + \[Alpha] m[
r]^(-2 + \[Alpha]) Derivative[1][m][r]^2 + 
m[r]^(-1 + \[Alpha]) (m'')[r]][
Function[r, (\[Rho][r])^(1/(\[Alpha] + 1/2))]] // FullSimplify
}
\end{ART}	
The remaining assertions follow by straightforward calculations.
\begin{ART}
{"solving the Hamilton-Jacobi equation for \rho'' and simplifying",
Assuming[Element[\[Alpha], Integers], 
FullSimplify[
Solve[1/(2 (1 + 2 \[Alpha])) r^(-1 - 2 d) \[Rho][
r]^(-1 - 
4/(1 + 2 \[Alpha])) (-2 r^(1 + 2 d) (1 + 2 \[Alpha]) (H + 
g[\[Rho][r]^(1/(1/2 + \[Alpha]))] - V[r]) \[Rho][
r]^(1 + 4/(1 + 2 \[Alpha])) + 
j^2 r^3 (1 + 2 \[Alpha]) \[Rho][
r] (\[Rho][r]^(1/(1/2 + \[Alpha])))^\[Alpha] - 
4 j r^(2 + 
d) \[Alpha] (\[Rho][
r]^(1/(1/2 + \[Alpha])))^(1 + \[Alpha]) Derivative[
1][\[Rho]][r] + 
4 r^(2 d) \[Rho][
r]^(4/(1 + 2 \[Alpha])) (\[Rho][
r]^(1/(1/2 + \[Alpha])))^\[Alpha] ((-1 + 
d) Derivative[1][\[Rho]][r] + r (\[Rho]'')[r])) == 
0, (\[Rho]'')[r]]]]
	}
\end{ART}
\begin{ART}
{"Checking the variational derivative of the functional in terms of \rho", 
Needs["VariationalMethods`"];
VariationalD[
r^(d - 1) (\[Rho]'[r]^2 + G[\[Rho][r]^2]/2 - V[r] \[Rho][r]^2/2 + 
j^2 r^(2 - 2 d)/(4 \[Rho][r]^2)), \[Rho][r], 
r]/(2 r^(d - 1)) // Expand	
	}
\end{ART}
\begin{ART}
{"The last variational principle", 
	VariationalD[
	r^(d - 1) (\[Rho]'[
	r]^2 + (2 \[Alpha] + 1) G[\[Rho][
	r]] + (2 \[Alpha] + 1)^2/ (2 (\[Alpha] + 1)) \[Rho][
	r]^((2 \[Alpha] + 2)/(2 \[Alpha] + 1)) (H - V (r))), \[Rho][
	r], r]/(2 r^(-1 + d)) // Expand
}
\end{ART}

\end{proof}
\begin{figure}[htb]
\centering
\begin{subfigure}[b]{\sizefigure\textwidth}
\includegraphics[width=\textwidth]{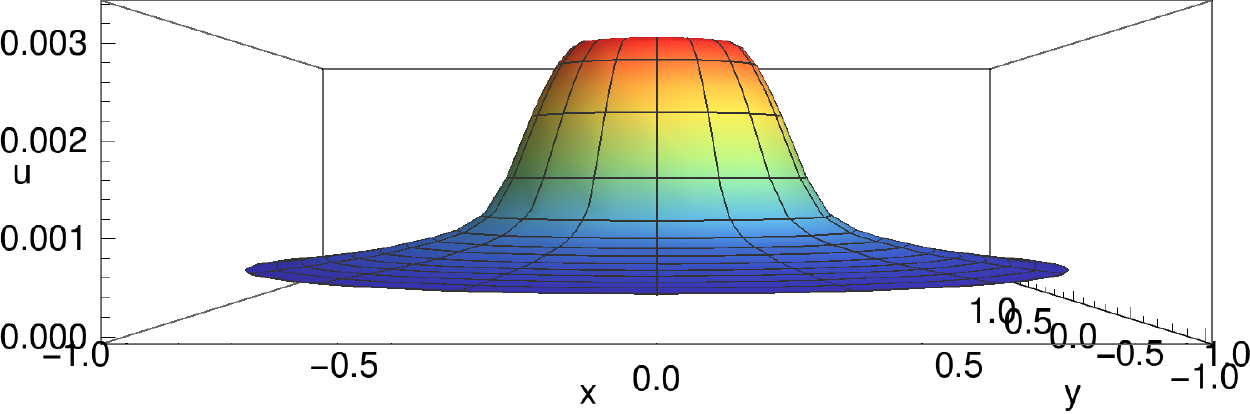}
\caption{Density $m$}
\label{fig:Plot3Dma3Front}
\end{subfigure}
\begin{subfigure}[b]{\sizefigure\textwidth}
\includegraphics[width=\textwidth]{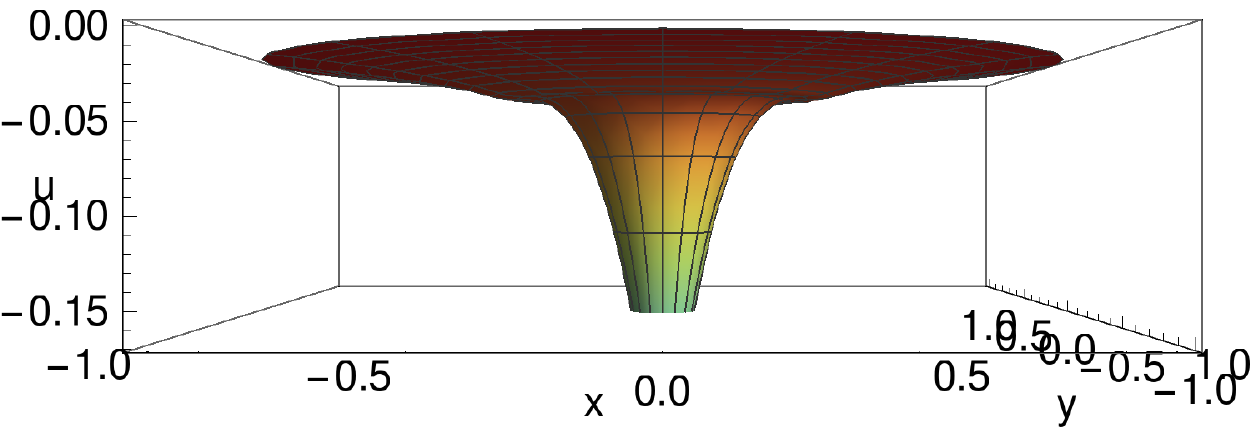}
\caption{Value function $u$}
\label{fig:Plot3Dma3Above}
\end{subfigure}
\begin{subfigure}[b]{\sizefigure\textwidth}
\includegraphics[width=\textwidth]{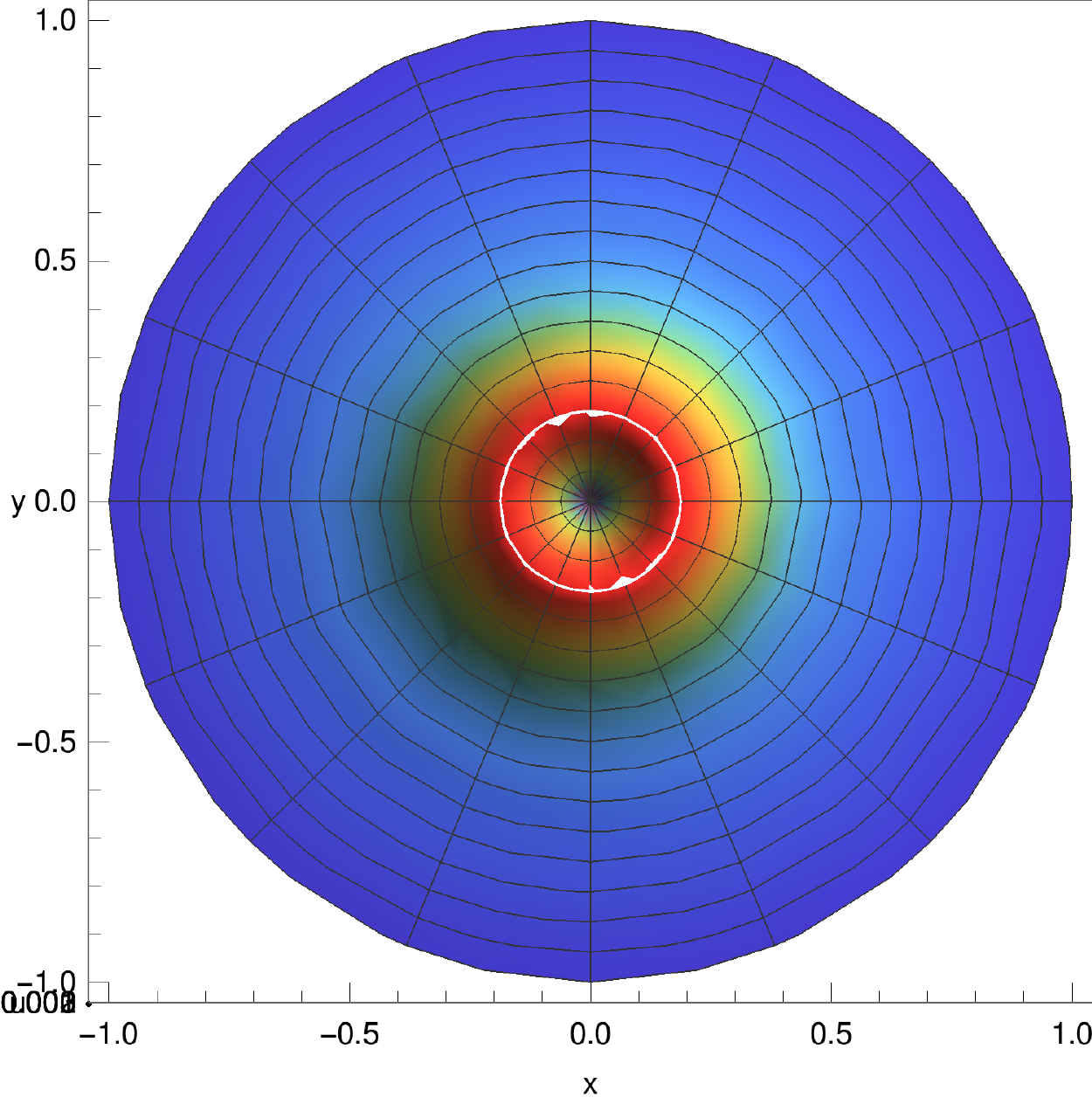}
\caption{Density $m$}
\label{fig:Plot3Dua3Front}
\end{subfigure}
\begin{subfigure}[b]{\sizefigure\textwidth}
\includegraphics[width=\textwidth]{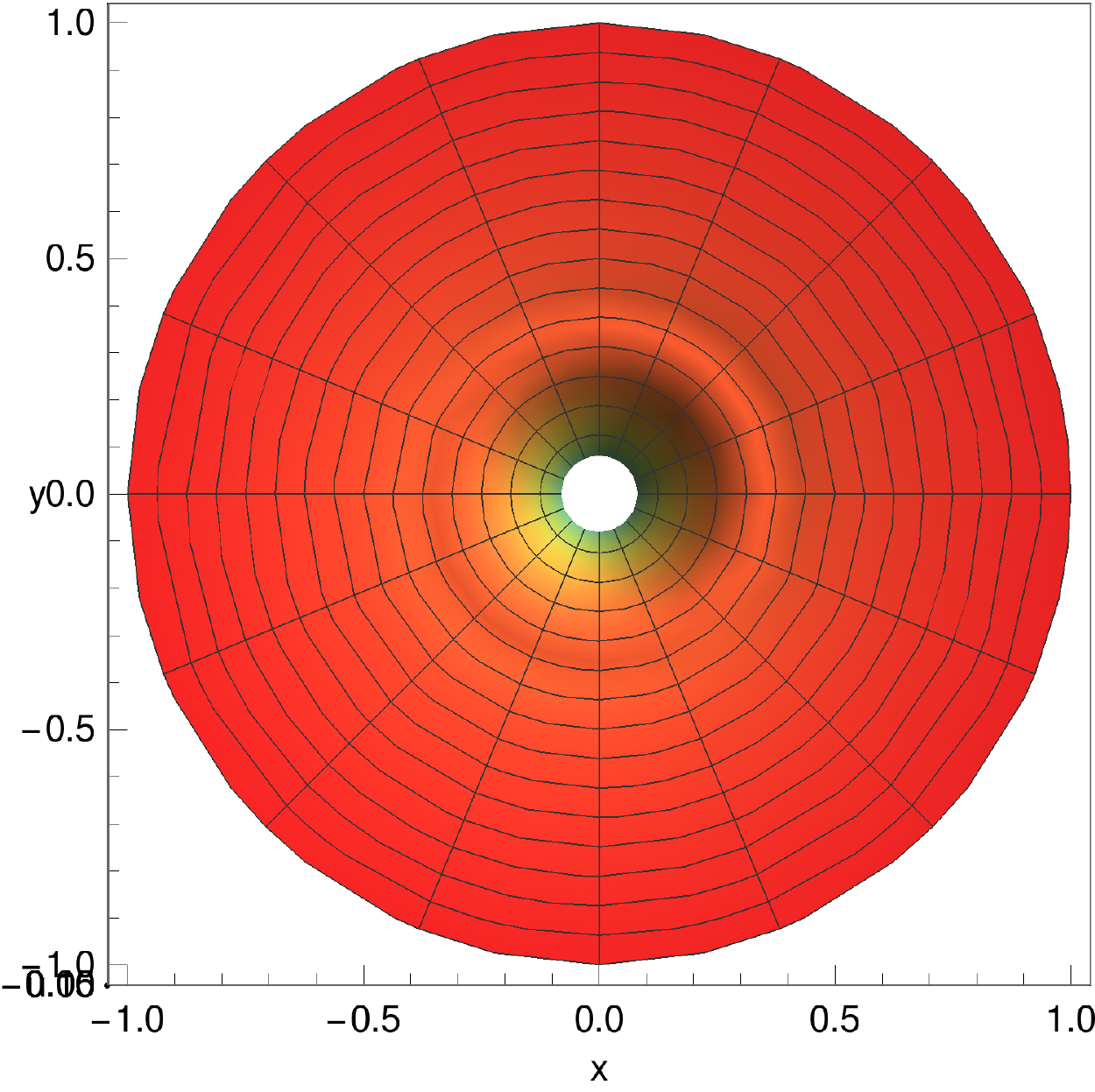}
\caption{Value function $u$}
\label{fig:PlotDua3Above}
\end{subfigure}
~ 
\caption{Numerical solution of Problem \ref{P1} for $d=2$, $g(m) = m$, $j=1$, $V(x)=e^{-|x|^2/2}\sin(\pi \left(|x|+\frac{1}{4})\right) $ and  $\alpha=1.3$.}\label{fig:solFirstOrderMFGPowerSinR}
\end{figure}
\begin{remark}
	The reason \eqref{eq:ODErhoj} admits a variational formulation when $\alpha=0$ or $j=0$ is that the coefficient of $\rho'(r)$ in \eqref{eq:ODErhoj} does not depend on $\rho(r)$ for this choice of parameters.
\end{remark}

\section{Final remarks}

Here, we study new examples of MFGs that can be solved explicitly. These examples 
are relevant to the understanding of qualitative features of MFGs and for the validation of numerical methods. 
For first-order problems, our solutions are explicit, up to the computation of integrals and the solution of algebraic equations. Further, 
they give insight on properties of the potential that are needed for the existence of a
global solution. For example, the condition \eqref{eq:int=1} imposes an asymptotic condition on $V$
without which solutions will not exist. For second-order equations, we construct new reduced models
and variational principles. These variational principles may prove useful in the numerical computation of
solutions, an issue that we plan to address in the future. 

Naturally, many questions remain open, in particular, because our models do not fit the standard theory for mean-field games in a compact state space. 
For example, are the solutions computed here stable under small perturbations? Do we have long-time convergence to stationary solutions? Are the conditions in \eqref{eq:int=1} sufficient for the existence of a solution in the non-radial case? All of these are challenging questions that we hope will be answered in the  near future.

\bibliographystyle{alpha}

\def\polhk#1{\setbox0=\hbox{#1}{\ooalign{\hidewidth
			\lower1.5ex\hbox{`}\hidewidth\crcr\unhbox0}}} \def\cprime{$'$}

\end{document}